%% file: isolation.tex
\documentclass[12pt,reqno]{amsart}

\usepackage[arrow,matrix,curve]{xy}

\usepackage[dvips]{graphicx} 
\usepackage{mathtools}

\usepackage{amssymb, latexsym, amsmath, amscd, array, esint, yhmath,
}

\usepackage{tikz}

\definecolor{Black}{RGB}{0, 0, 0}

\usepackage[breaklinks,colorlinks=true,allcolors=Black]{hyperref}

\usepackage{breakurl}   

\makeatletter
\@namedef{subjclassname@2020}{\textup{2020} Mathematics Subject Classification}
\makeatother

\newtheorem{theoremABC}{Theorem}

\newtheorem{theorem}{Theorem}[section]
\newtheorem{lemma}[theorem]{Lemma}

\newtheorem{proposition}[theorem]{Proposition}

\theoremstyle{definition}

\newtheorem{example}[theorem]{Example}
\newtheorem{remark}[theorem]{Remark}

\numberwithin{equation}{section}

\newcommand\N {{\mathbb N}} 
\newcommand\C {{\mathbb C}}
\newcommand\R {{\mathbb R}}

\newcommand\T {{\mathbb T}} 
\newcommand\K {{\mathbb K}}

\DeclareMathOperator{\area}{{\rm area}}

\DeclareMathOperator{\sys}{{\rm sys}}

\newcommand{\RP}{{\mathbb R\mathbb P}}

\long\def\forget#1\forgotten{} %

\numberwithin{equation}{section}

\title[Klein bottle and optimal systolic inequality] {Klein bottle and
  optimal systolic inequality for nonpositively curved surfaces}

\author[M. Katz]{Mikhail G. Katz} \address{Department of
  Mathematics, Bar Ilan University, Ramat Gan 5290002 Israel}
\email{katzmik@math.biu.ac.il}

\author[S.\;Sabourau]{St\'ephane Sabourau}
\address{\parbox{\linewidth}{Univ Paris Est Creteil, CNRS, LAMA, F-94010 Creteil, France \\
Univ Gustave Eiffel, LAMA, F-77447 Marne-la-Vall\'ee, France}}
\email{stephane.sabourau@u-pec.fr}

\subjclass[2020]{Primary 53C20; Secondary 53C23}

\keywords{systole, systolic area, nonpositively curved surfaces}

\begin{document}

\begin{abstract}
We show that the systolic area of every nonpositively curved closed
surface~$M$ other than the torus is at least~$1$, with equality if and
only if $M$ is isometric to a square flat Klein bottle.  The proof
focuses on the Klein-double $4\RP^2$ and exploits Weil's isoperimetric
inequality and a comparison theorem involving a new kind of
exponential-type map.
\end{abstract}

\maketitle

\section{Introduction}

The systole of a closed nonsimply connected surface~$M$ with a
Riemannian metric is defined as the minimal length of a
noncontractible closed curve.  It is realized by the length of a
noncontractible closed geodesic.  The systolic area of~$M$ is then
defined as
\[
\sigma(M) = \frac{\area(M)}{\sys(M)^2}.
\]


Pu's systolic inequality~\cite{pu} asserts that the minimal systolic area of the projective plane~$\RP^2$ equals~$\frac{2}{\pi}$ and is attained precisely by the round metrics.
Gromov~\cite{gro-frm} later proved that the systolic area of every closed surface other than~$S^2$ or~$\RP^2$ is at least~$\frac{3}{4}$.
Together, these results imply that the systolic area of a closed nonsimply connected surface is at least~$\frac{2}{\pi}$, with equality if and only if the surface is a round projective plane.

For arbitrary metrics, optimal systolic inequalities are notoriously difficult to establish.
They are known for only three closed surfaces: the torus (see~\cite{katz-book}), the projective plane~\cite{pu} and the Klein bottle~\cite{bav}.
For noncompact surfaces, analogous optimal inequalities have recently been obtained for the three- and four-punctured plane~\cite{JS}, where the systole is replaced by the length of the shortest noncontractible closed geodesic.

Despite this progress, fundamental questions remain open: for example, it is still unknown whether every genus~$g>1$ surface satisfies the same systolic inequality as Loewner's on the torus, even for $g=3$ (see~\cite{KS05}, \cite{KS06b} and~\cite{KS12} for results in this direction).

One might expect that the larger the absolute value of the Euler characteristic of a surface, the larger its minimal systolic area should be.  
While it is unclear whether such monotonicity holds, it is known that the systolic area tends to infinity with the absolute value of the Euler characteristic; see~\cite{gro-frm} and~\cite{KS05}.

\medskip

Recent work has focused on systolic inequalities under curvature
assumptions.  For nonpositively curved metrics, we showed
in~\cite{KS24} that the systolic area of every closed nonpositively
curved surface is at least~$\frac{\sqrt{3}}{2}$, with equality if and
only if the surface is a flat hexagonal torus.

\medskip

In such a nonpositively curved setting, we prove that the next possible
value of the minimal systolic area is realized by the square flat
Klein bottle.

\begin{theoremABC} \label{theo:main}
Each nonpositively curved closed surface $M$ other than the torus
satisfies $\sigma(M) \geq 1$, i.e.,
\[
\area(M)\geq\sys(M)^2,
\]
with equality if and only if $M$ is isometric to a square flat Klein
bottle.
%
\end{theoremABC}

To summarize, the first minimal systolic area among nonpositively
curved surfaces is~$\tfrac{\sqrt{3}}{2}$, attained uniquely by the
flat hexagonal tori.  The second minimal systolic area is~$1$,
realized uniquely by the square flat Klein bottles.  The value of the
third minimal systolic area remains unknown, although it is bounded
below by~$1.002$; see Theorem~\ref{BB}.  A natural candidate is the
systolically extremal nonpositively curved metric on the connected
sum~$3\RP^2$ of three copies of~$\RP^2$ -- also known as Dyck's
surface -- constructed in~\cite{KS15}, whose systolic area is
$1+\frac{(169-38\sqrt{19})^{1/2}}{12} \simeq 1.15$.  This metric is
flat with finitely many conical singularities, a feature shared by all
systolically extremal nonpositively curved surfaces by~\cite{KS21}.

\medskip

The proof of Theorem~\ref{theo:main} proceeds in three steps:

First, it builds on the existence of a
large-area disk $D\subseteq M$ established in~\cite{KS24}, which follows from an averaging argument using the invariance of the Liouville measure under the geodesic flow.  Combined
with earlier systolic inequalities for the genus-two
surface~\cite{KS06} and for Dyck's surface~\cite{KS15}, this result
allows us to narrow down the problem to the case of the surface
$4\RP^2$ of Euler characteristic~$-2$.  Since this surface is
homeomorphic to the double of a Klein bottle, we will refer to it as
the \emph{Klein-double}.  

Second, we study the polygonal Voronoi cell concentric
with the large-area disk, control the number of its edges, and then
apply a comparison inequality in nonpositive curvature as
in~\cite{KS06} to enhance the previous bound.  

Third, while such an argument alone does not suffice, we develop a new
kind of exponential-type map, and exploit it in combination with
Weil's isoperimetric inequality in the Cartan--Hadamard plane.  The
new exponential-type map is noncontracting, and takes advantage of the
extra boundary length of the large-area disk, enabling us to further
boost the estimate on the surrounding area -- pushing it beyond the
Klein bottle bound.

\medskip

Actually, we are able to obtain the following slightly stronger bound
than that of Theorem~\ref{theo:main}.

\begin{theoremABC}
\label{BB}
A closed surface $M$ of negative Euler characteristic with a
nonpositively curved metric satisfies $\sigma(M) > 1.002$.
\end{theoremABC}
See Theorem~\ref{theo:B} for a sharper statement.

\medskip

\forget

The proof of Theorem~\ref{theo:main} builds on the existence
of a large-area disk $D\subseteq M$ established in~\cite{KS24}, which
follows from an averaging argument using the invariance of the
Liouville measure under the geodesic flow.  Combined with earlier
systolic inequalities for the genus-two surface~\cite{KS06} and for
Dyck's surface~\cite{KS15}, this result allows us to narrow down the
problem to the case of the surface $4\RP^2$ of Euler
characteristic~$-2$.  Since this surface is homeomorphic to the double
of a Klein bottle, we will refer to it as the \emph{Klein-double}.
For this surface, we are able to obtain a slightly stronger bound than
that of Theorem~\ref{theo:main}.

\begin{theoremABC}
\label{BB}
Let $M$ be a closed surface of negative Euler characteristic with a
nonpositively curved metric.  Then $\sigma(M) > 1.002$.
\end{theoremABC}
See Theorem~\ref{theo:B} for a sharper statement.

The general idea is to study the polygonal Voronoi cell concentric
with the large-area disk, control the number of its edges, and then
apply a comparison inequality in nonpositive curvature as
in~\cite{KS06}.  While such an argument alone does not suffice, we
develop a new kind of exponential-type map, and exploit it in
combination with the isoperimetric inequality in the Cartan--Hadamard
plane.  The new exponential-type map is noncontracting, and takes
advantage of the extra boundary length of the large-area disk,
enabling us to obtain sharper estimates on the surrounding area.

\forgotten

We conclude the article by presenting two nonpositively curved
piecewise flat metrics on the Klein-double, exploiting a hyperelliptic
curve of genus $3$ whose equation goes back to Klein's \emph{Lectures
on the Icosahedron}, originally published in 1884 (see \cite[II,
  \S12]{Kl56} where the corresponding homogeneous form is denoted
$W$).  We deduce the following.

\begin{theoremABC}
The minimal systolic area among all nonpositively curved metrics on
the Klein-double lies between~$1.002$ and~$1.37$.
\end{theoremABC}

See Proposition~\ref{p51} for further details. 

\medskip

\noindent {\bf Note.} 

After the galley proofs were returned to the journal, the
preprint~\cite{CV} appeared on arXiv, leading to sharper systolic
inequalities without any curvature assumption.  In particular, the
systolic area of every surface of negative Euler characteristic is
greater than 1 and every genus $g\geq1$ surface is Loewner.

\section{Narrowing down the Euler characteristic} \label{sec:two}



Let $M$ be a connected closed surface with a nonpositively curved
metric, which is not homeomorphic to the torus~$\T^2$ or the Klein
bottle~$\K^2$.  Thus, the Euler characteristic~$\chi(M)$ of~$M$ is
negative.  In~\cite{KS24}, we proved that every nonpositively curved
closed surface admits a suitable disk $D \subseteq M$ of
radius~$\frac{1}{2} \sys(M)$ such that
\begin{equation} \label{eq:D}
\area(D) \geq \left( \frac{\pi}{4} + \frac{\pi^2 |\chi(M)|}{96} \, \frac{\sys(M)^2}{\area(M)} \right) \sys(M)^2.
\end{equation}
Using the bound $\area(M) \geq \area(D)$, this yields the quadratic
relation
\[
\sigma(M)^2 - \frac{\pi}{4} \, \sigma(M) - \frac{\pi^2 |\chi(M)|}{96} \geq 0
\]
for the systolic area~$\sigma(M)$ of~$M$.
Hence,
\begin{equation} \label{eq:chi}
\sigma(M) \geq \frac{\pi}{8} + \frac{\pi}{8} \, \sqrt{1 + \frac{2}{3} \, | \chi(M) |}.
\end{equation}
When $|\chi(M)| \geq 3$, we deduce the lower bound
\begin{equation} \label{eq:8}
\sigma(M) \geq \frac{\pi}{8} (1+\sqrt{3}) \simeq 1.07
\end{equation}
as desired.  For the remaining values of the Euler characteristic,
namely $\chi(M)=-1$ or $-2$, the inequality~\eqref{eq:chi} does not
provide a sufficient lower bound on the systolic area.  These values
of the Euler characteristic correspond to the genus two
surface~$\Sigma_2$, Dyck's surface which is the connected
sum~$3\RP^2$, and the Klein-double~$4\RP^2$.  For the genus two
surface and Dyck's surface, we established optimal systolic
inequalities among nonpositively curved metrics in~\cite{KS06}
and~\cite{KS15}.  Namely,
\[
\sigma(\Sigma_2) \geq 3 (\sqrt{2}-1) \simeq 1.24
\]
and
\[
\sigma(3 \RP^2) \geq 1 + \frac{(169-38\sqrt{19})^{1/2}}{12} \simeq
1.15.
\]
It follows that all surfaces~$M$, except possibly for the
Klein-double, have a systolic area greater than~$1.07$.  Thus, it
remains to treat the case of the Klein-double of Euler
characteristic~$-2$.

\medskip

In this case, we can prove a stronger systolic inequality than Theorem~\ref{theo:main}, valid for all surfaces of negative Euler characteristic.

\begin{theorem} \label{theo:B}
A closed surface $M$ of negative Euler characteristic~$\chi$ with a
nonpositively curved metric satisfies
\[
\sigma(M) \geq \sigma_\chi
\]
where $\sigma_\chi > 1.002$ is the solution of the equation~\eqref{eq:sigma0}.
\end{theorem}

The two following sections are devoted to a proof of this
result.


\section{Voronoi cell and area lower bound}

We argue by contradiction.  Specifically, we can assume that
$\sigma(M) < \sigma_\chi$, where $\sigma_\chi > 1.002$ is the solution
of the equation~\eqref{eq:sigma0}.

For the purpose of the inequality of Theorem~\ref{theo:main}, it is
enough to prove Theorem~\ref{theo:B} for the Klein-double using the
value $1$ in place of $\sigma_\chi$.

\subsection{Voronoi cell centered around~$D$}

\mbox{ } \medskip

\forget
From now on, we will argue by contradiction.
Specifically, we can assume that $M=4 \RP^2$ and that $\sigma(M) \leq \sigma_0$, where 
\[
\sigma_0=\frac{\pi}{8}(1+\sqrt{3})=1.07...
\]
If we are only interested in the first part of the theorem, we can take~$\sigma_0=1$.
\forgotten

By~\eqref{eq:D}, the large-area disk~$D \subseteq M$ provided by~\cite{KS24} satisfies

\begin{equation} \label{eq:48}
\area(D) > \left( \frac{\pi}{4} + \frac{\pi^2 |\chi|}{96 \,
  \sigma_\chi} \right) \, \sys(M)^2.
\end{equation}
Note that when $M$ is the Klein-double and $\sigma_\chi$ is replaced
by $1$, the multiplicative constant $\frac{\pi}{4} + \frac{\pi^2}{48}
\simeq 0.991$ falls short of $1$ by less than $1\%$.

Now, the idea is to seek a lower bound for the area of the region
of~$M$ outside the disk~$D$.  Let $D_i \subseteq \bar{M},\, i \in \N$
be the lifts of~$D$ to the universal cover~$\bar{M}$ of~$M$.  Denote
by~$x_i$ the center of~$D_i$.  Consider the Voronoi cell~$V_i$
around~$x_i$ defined by
\[
V_i = \{ x \in \bar{M} \mid d(x,x_i) \leq d(x,x_j) \textrm{ for every } j \neq i \}.
\]

Such a Voronoi cell on~$\bar{M}$ is a star-shaped polygon whose edges are arcs of the equidistant curves between a pair of points $\{x_i,x_j\}$.
Note that the edges are not necessarily geodesics.
Each Voronoi cell~$V$ is a fundamental domain of~$M$ and its edges are pairwise identified in~$M$ under the projection~$\pi:\bar{M} \to M$.
This yields a cellular decomposition of~$M$ whose $1$-skeleton coincides with the projection of the boundary~$\partial V$ of~$V$.
The Euler characteristic of~$M$ can be expressed in terms of the numbers~$e$ and~$v$ of edges and vertices of the cellular decomposition.
Specifically,
\[
\chi= v-e+1.
\]
Combined with the classical inequality $2e \geq 3v$, we deduce the following.

\begin{lemma}
\label{l31}
The number~$n$ of edges of the Voronoi cell\,~$V$ satisfies
\[
n=2e \leq 6 |\chi| +6.
\]
\end{lemma}
This upper bound on $n$ yields a lower bound for the area of the
region outside $D$.

\subsection{Euclidean tangent polygon~$P$} \label{sec:P}

\mbox{ } \medskip

This section and the next describe an attempt to derive the desired area lower bound, building on an approach first introduced in~\cite{KS06}.
While this initial attempt falls short of the desired bound, the subsequent sections show how to adapt the method to achieve the desired result. 

\medskip

Renumbering the indices if necessary, we can assume that
$V_1,\dots,V_n$ are the $n$ Voronoi cells sharing an edge with~$V_0$.
Let $\mathcal{T}_E= T_{x_0} \bar{M}$.  The subscript~$E$ is here to
emphasize that the tangent plane is endowed with a Euclidean metric.
Denote by~$\omega_i \in \mathcal{T}_E$ the preimage of~$x_i$ under the
exponential map $\exp_{\mathcal{T}_E}^{\phantom{I^I_I}}: \mathcal{T}_E
\to \bar{M}$ with $1 \leq i \leq n$.  Note that $\omega_0$ is the
origin of~$\mathcal{T}_E$.  The equidistant line in~$\mathcal{T}_E$
between~$\omega_0$ and~$\omega_i$ is denoted by
\[
\Delta_i \subseteq \mathcal{T}_E.
\]

Let $P \subseteq \mathcal{T}_E$ be the Euclidean polygon formed by the
intersection of the halfplanes containing the origin~$\omega_0$ and
bounded by the lines~$\Delta_i$ with $1 \leq i \leq n$.  Consider a
point~$\omega \in P$ and its image
$y=\exp_{\mathcal{T}_E}^{\phantom{I^I_I}}(\omega)$ in~$\bar{M}$.  By
  Alexandrov's comparison inequality~\cite[\S II]{BH},
the exponential map does not decrease distances on a nonpositively curved plane.
Therefore,
\[
d_{\bar{M}}(x_0,y) = d_{\mathcal{T}_E}(\omega_0,\omega) \leq d_{\mathcal{T}_E}(\omega_i,\omega) \leq d_{\bar{M}}(x_i,y).
\]
It follows that the exponential image of~$P$ is contained in the Voronoi cell~$V_0$.
By Gauss' lemma, we also have that the exponential image of the disk $D(\omega_0,\frac{1}{2} \sys(M)) \subseteq \mathcal{T}_E$ coincides with~$D_0$.
Since the exponential map is distance nondecreasing, and thus area nondecreasing, we have
\begin{equation} \label{eq:V0-D0}
\area(V_0 \setminus D_0) \geq \area\big(P \setminus D(\omega_0,\tfrac{1}{2} \sys(M))\big) \geq \area(P) - \frac{\pi}{4} \sys(M)^2.
\end{equation}

\subsection{Area lower bound on~$P$} \label{sec:PP}

\mbox{ } \medskip

The polygon~$P$ is a compact Euclidean polygon with $k$ sides, which can be decomposed into $k$ triangles over these $k$ sides with angle~$\theta_j$ at~$\omega_0$.
Since $d_{\mathcal{T}_E}(\omega_0,\omega_i) = d_{\bar{M}}(x_0,x_i) \geq \sys(M)$, the area of each of these triangles is bounded from below by
\begin{equation} \label{eq:tan}
\left( \frac{\sys(M)}{2} \right)^2 \tan \left( \frac{\theta_j}{2} \right).
\end{equation}
By Jensen's inequality applied to the convex function $\tan(\frac{x}{2})$ when $0<x<\pi$, we deduce that
\begin{equation} \label{eq:P}
\area(P) \geq \frac{1}{4}  \sum_{j=1}^k \tan \left( \frac{\theta_j}{2} \right) \sys(M)^2 \geq \frac{k}{4}  \tan \left( \frac{\pi}{k} \right)  \sys(M)^2.
\end{equation}
It follows from~\eqref{eq:48}, \eqref{eq:V0-D0}, \eqref{eq:P} and the
bound $k \leq n \leq 6|\chi|+6$ of Lemma~\ref{l31} that
\begin{align*}
\area(M) & = \area(D_0) + \area(V_0 \setminus D_0) \\
 & \geq \left( \frac{\pi^2 |\chi|}{96 \, \sigma_\chi} + \frac{3}{2} (|\chi|+1) \tan \frac{\pi}{6|\chi|+6} \right) \sys(M)^2.
\end{align*}
This is still not enough to conclude when~$M$ is the Klein-double even
when~$\sigma_\chi$ is replaced by $1$, since
the multiplicative constant equals $\frac{\pi^2}{48} + \frac{9}{2}
\tan \left(\frac{\pi}{18}\right) \simeq 0.99908$, which falls short of
$1$ by less than~$1\text{\textperthousand}$.  We develop a sharper
estimate in Section~\ref{s4}.


\section{Isoperimetric inequality and conical tangent plane}
\label{s4}


We continue with the proof of Theorem~\ref{theo:B}, keeping the same
notation.  The idea at this stage is to use the isoperimetric
inequality to show that the circumference of the large-area disk~$D$
exceeds that of the Euclidean disk of the same radius, thereby
yielding extra area around~$D$.

\medskip

By Weil's isoperimetric inequality on a Cartan--Hadamard plane~\cite[Theorem~2.3]{rit23}, the circumference~$L$ of~$D_0 \subseteq \bar{M}$ satisfies 
\[
L^2 \geq 4\pi \, \area(D_0).
\]
Combined with~\eqref{eq:48}, this implies the lower bound
\[
L \geq \pi \sqrt{1+\frac{\pi |\chi|}{24 \, \sigma_\chi}} \, \sys(M).
\]

Let us show how to improve the lower bound on the area of~$V_0 \setminus D_0$ given by~\eqref{eq:V0-D0} and~\eqref{eq:P} using the lower bound on~$L$.
Denote by~$\mathcal{T}_\theta$ the tangent plane~$T_{x_0} \bar{M}$ endowed with the flat metric with a conical singularity of angle
\begin{equation} \label{eq:theta}
\theta = 2 \pi \sqrt{1+\frac{\pi |\chi|}{24 \, \sigma_\chi}} > 2\pi
\end{equation}
at its origin~$\omega_0$.

With this choice of angle, the circumference of the disk $D_\theta \subseteq \mathcal{T}_\theta$ of radius~$\frac{1}{2} \sys(M)$ centered at~$\omega_0$ is bounded by~$L$.
Note also that
\begin{equation} \label{eq:D_theta}
\area(D_\theta) = \frac{\theta}{8} \sys(M)^2. 
\end{equation}

\subsection{Exponential-type map}

\mbox{ } \medskip

We will need a modified version of the exponential map, in which the polar coordinates in the tangent plane~$\mathcal{T}$ are replaced by analogous coordinates in the conical tangent plane with a disk removed~$\mathcal{T}_\theta \setminus D_\theta$.

\medskip

Since the circle~$\partial D_\theta$ is no longer than~$\partial D_0$, there exists a length-nondecreasing identification given by a reparametrization proportional to arclength
\[
  \begin{aligned}
    \Phi_\theta : \partial D_\theta & \longrightarrow \partial D_0 \\
    \nu & \longmapsto \bar{\nu}.
  \end{aligned}
\]
Let us extend this map into an exponential-type map supported
outside~$D_0$ (and not around~$x_0$)
\[
\Phi_\theta: \mathcal{T}_\theta \setminus D_\theta \to \bar{M} \setminus D_0
\]
as follows.
Given a vector $\eta \in \mathcal{T}_\theta \setminus D_\theta$, denote by~$h$ the distance between~$\eta$ and its radial projection~$\nu$ to the circle~$\partial D_\theta$.
We will refer to $(h,\nu)$ as the polar-like coordinates of~$\eta$.
Its image~$\Phi_\theta(\eta)$, also denoted~$\bar{\eta}$, is defined as the point on the ray~$[\omega_0,\nu)$ at distance~$h$ from~$\bar{\nu}$.

\begin{lemma} \label{lem:noncontracting}
The map $\Phi_\theta: \mathcal{T}_\theta \setminus D_\theta \to \bar{M} \setminus D_0$ is locally distance-nondecreasing.
\end{lemma}

\begin{proof}
By construction, the map~$\Phi_\theta$ is distance-preserving along rays arising from the origin of~$\mathcal{T}_\theta$.
By Gauss' lemma, we only need to measure the dilatation effect of the differential of~$\Phi_\theta$ on the vectors~$\mu$ orthogonal to these rays.

Let $\mu$ be a vector with basepoint~$\eta_0 \in \mathcal{T}_\theta$ tangent to a ray arising from the origin of~$\mathcal{T}_\theta$.
Denote by~$(h_0,\nu_0)$ the polar-like coordinate of~$\eta_0$.
Consider a one-parameter family of vectors~$\eta_t$ with polar-like coordinate~$(h_0,\nu_t)$ such that $\frac{d\nu_t}{dt}_{|t=0}=\mu$.
Denote by~$\bar{\eta}_t$ the $\Phi_\theta$-image of~$\eta_t$.

Recall that $\nu_t$ and $\eta_t$ are at distance~$\frac{1}{2} \sys(M)$ and~$\frac{1}{2} \sys(M)+h$ from the vertex of~$\mathcal{T}_\theta$. 
By Thales' theorem, this yields
\begin{equation} \label{eq:conv1}
d_{\mathcal{T}_\theta}^{\phantom{I}}(\nu_0, \nu_t) = \frac{\tfrac{1}{2}
  \sys}{\tfrac{1}{2} \sys + h} \, d_{\mathcal{T}_\theta}^{\phantom{I}}(\eta_0,
\eta_t).
\end{equation}

Recall also that $\bar{\nu}_t$ and $\bar{\eta}_t$ are at distance respectively~$\frac{1}{2} \sys(M)$ and~$\frac{1}{2} \sys(M)+h$ from~$x_0$.
Thus, by convexity of the distance function in complete nonpositively curved space (see~\cite[\S II.2.2]{BH} and Figure~\ref{fig:div}),
we derive
\begin{equation} \label{eq:conv2}
d_{\bar{M}}(\bar{\nu}_0, \bar{\nu}_t) \leq \frac{\tfrac{1}{2} \sys } {\tfrac{1}{2} \sys + h}\, d_{\bar{M}}(\bar{\eta}_0, \bar{\eta}_t).
\end{equation}

\begin{figure}[htbp!] 
\vspace{2.2cm}
\def\svgwidth{3.5cm}
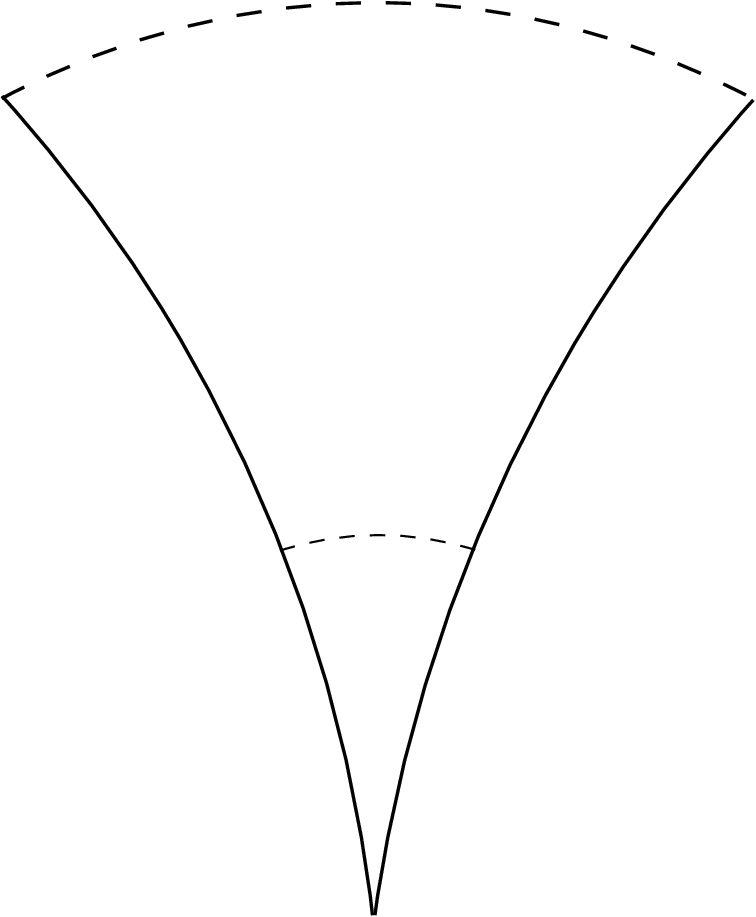 
\vspace{0.3cm}
\caption{Geodesic divergence} \label{fig:div}
\end{figure}

By construction, the exponential-type map is length-nondecreasing on~$\partial D_\theta$.
Thus, $d_{\mathcal{T}_\theta}(\nu_0, \nu_t) \leq d_{\bar{M}}(\bar{\nu}_0, \bar{\nu}_t)+o(t)$.
Combined with the relations~\eqref{eq:conv1} and~\eqref{eq:conv2}, this yields
\[
d_{\mathcal{T}_\theta}(\eta_0, \eta_t) \leq d_{\bar{M}}(\bar{\eta}_0, \bar{\eta}_t) + o(t).
\]
Taking the Taylor expansions of these expressions in this inequality, we derive
\[
\lVert \mu \rVert \leq \lVert d\Phi_\theta(\mu) \rVert
\]
as desired.
\end{proof}

\subsection{Conical tangent polygon~$Q$}

\mbox{ } \medskip

To complete the proof of Theorem~\ref{theo:B}, we adapt and refine
the arguments from Sections~\ref{sec:P} and~\ref{sec:PP}.

\medskip

As with the definition of~$P$, let~$Q \subseteq \mathcal{T}_\theta$ be the polygon formed by the intersection of the regions containing the origin~$\omega_0$ and bounded by the equidistant lines~$\Delta_i^\theta \subseteq \mathcal{T}_\theta$ between~$\omega_0$ and~$\omega_i$.
The polygon~$Q$ is flat with a conical singularity of angle~$\theta>2\pi$ at~$\omega_0$.

We note that the exponential-type map is not necessarily
distance-nondecreasing at a global scale.  Therefore we need to refine
the argument of Section~\ref{sec:P} leading to the following result.

\begin{lemma}
The image of~$Q$ by the exponential-type map lies within the Voronoi cell~$V_0$.
\end{lemma}

\begin{proof}
Consider a point~$\omega \in Q$ and its image $y=\Phi_\theta(\omega)$ in~$\bar{M}$.
By construction, the point~$\omega$ lies in the region containing the origin~$\omega_0$ and bounded by the equidistant lines~$\Delta_i^\theta$ in~$\mathcal{T}_\theta$.
Therefore,
\[
d_{\bar{M}}(x_0,y) = d_{\mathcal{T}_\theta}(\omega_0,\omega) \leq d_{\mathcal{T}_\theta}(\omega_i,\omega).
\]

Suppose the minimizing arc~$\alpha \subseteq \bar{M}$ between~$x_i$ and~$y$ avoids~$D_0$.
Then its preimage~$\Phi^{-1}_\theta(\alpha) \subseteq \mathcal{T}_\theta$ is an arc between~$\omega_i$ and~$\omega$, which is no longer than~$\alpha$ by Lemma~\ref{lem:noncontracting}.
Thus,
\[
d_{\bar{M}}(x_0,y) \leq d_{\mathcal{T}_\theta}(\omega_i,\omega) \leq d_{\bar{M}}(x_i,y).
\]
It follows that $y$ lies in the Voronoi cell~$V_0$ centered around~$x_0$.

Suppose the minimizing arc~$\alpha$ intersects~$D_0$.
Since $D_0$ is contained in the Voronoi cell~$V_0$, a point~$z$ of this arc is closer to~$x_0$ than~$x_i$.
Hence,
\[
d_{\bar{M}}(x_0,y) \leq d_{\bar{M}}(x_0,z) + d_{\bar{M}}(z,y) \leq d_{\bar{M}}(x_i,z) + d_{\bar{M}}(z,y) = d_{\bar{M}}(x_i,y).
\]
It follows as before that $y$ lies in~$V_0$.
\end{proof}

As a consequence of Lemma~\ref{lem:noncontracting}, the exponential-type map does not decrease areas.
With the relation~\eqref{eq:D_theta}, this yields the inequalities
\begin{equation} \label{eq:Q-}
\area(V_0 \setminus D_0) \geq \area(Q \setminus D_\theta) = \area(Q) - \frac{\theta}{8} \sys(M)^2.       
\end{equation}

Now, as in Section~\ref{sec:PP}, we decompose~$Q$ into at most~$n$ triangles with a common vertex at~$\omega_0$.
Applying the area comparison inequalities~\eqref{eq:tan}, Jensen's inequality and the fact that $\omega_0$ is a conical singularity of angle~$\theta$, we obtain as in Section~\ref{sec:PP} that
\begin{equation} \label{eq:Q}
\area(Q) \geq \frac{n}{4} \tan\left( \frac{\theta}{2n} \right) \sys(M)^2.
\end{equation}

We now combine the inequalities~\eqref{eq:48}, \eqref{eq:Q-}
and~\eqref{eq:Q} with the expression~\eqref{eq:theta} for~$\theta$ and
the bound~$n \leq 6|\chi|+6$.  Let
$r(\sigma)=\sqrt{1+\frac{\pi|\chi|}{24\,\sigma}}$.  Then
\begin{align}
\area(M) & = \area(D_0) + \area(V_0 \setminus D_0) \nonumber \\ & \geq
\left( \frac{\pi}{4} + \frac{\pi^2 |\chi|}{96 \, \sigma_\chi} +
\frac{3}{2} (|\chi|+1) \tan \frac{\pi\, r(\sigma_\chi)}{6|\chi|+6} -
\frac{\pi}{4} r(\sigma_\chi) \right) \sys(M)^2 \nonumber \\ & \geq
\left( \frac{\pi}{4} \left( r(\sigma_\chi)^2-r(\sigma_\chi) \right) +
\frac{3}{2} (|\chi|+1) \tan \frac{\pi\, r(\sigma_\chi)}{6|\chi|+6}
\right) \sys(M)^2.  \label{eq:last}
\end{align}

Now, the value of~$\sigma_\chi$ is chosen to be the (positive) solution of the equation
\begin{equation} \label{eq:sigma0}
\frac{\pi}{4} \left( r(\sigma_\chi)^2-r(\sigma_\chi) \right) + \frac{3}{2} (|\chi|+1) \tan \frac{\pi\, r(\sigma_\chi)}{6|\chi|+6} = \sigma_\chi
\end{equation}
which is greater than $1.002$.  It follows that $\sigma(M) \geq
\sigma_\chi$, which contradicts our assumption.  (When $M$ is the
Klein-double and $\sigma_\chi$ is replaced by $1$, the multiplicative
constant in~\eqref{eq:last} is roughly equal to~$1.099$, which is
greater than~$1$.)

\forget
\begin{remark}
Consider a connected closed surface~$M$ other than the torus and the Klein bottle with a nonpositively curved metric.
Suppose that $|\chi(M)| \geq 4$.
In this case, the systolic inequality~\eqref{eq:8} can be improved into
\[
\sigma(M) \geq \frac{\pi}{24} (3 +\sqrt{33}) \simeq 1.14.
\]
This, together with the discussion in Section~\ref{sec:two}, demonstrates that to replace the systolic inequality in the second part of Theorem~\ref{theo:main} with that of Theorem~\ref{theo:B}, it suffices to establish the inequality for the connected closed surface of Euler characteristic~$-3$, namely~$5 \RP^2$.
\end{remark}
\forgotten

\begin{remark}
Denote by $\tau_\chi = \frac{\pi}{8} + \frac{\pi}{8} \sqrt{1 + \frac{2}{3} | \chi|}$ the systolic lower bound given by~\eqref{eq:chi}.
While the systolic lower bound given by~$\sigma_\chi$ is better, the difference between the two bounds diminishes as the absolute value of~$\chi$ increases, as illustrated by Table~\ref{table}.

\begin{table}[htbp!]
\begin{tabular}{|c|c|c|c|c|c|}
\hline
$\chi$ & -2 & -3 & -4 & -5 & -6 \\
\hline
$\sigma_\chi$ & $1.002$ & $1.078$ & $1.148$ & $1.213$ & $1.272$ \\
\hline
$\tau_\chi$ & $0.992$ & $1.072$ & $1.144$ & $1.210$ & $1.270$ \\
\hline
\end{tabular}
\medskip
\caption{Approximate values of~$\sigma_\chi$ and~$\tau_\chi$} \label{table}
\end{table}
\end{remark}

\section{Some nonpositively curved metrics on the Klein-double}

In this section, we will present two nonpositively curved piecewise
flat metrics with conical singularities on the Klein-double~$4\RP^2$.
The first example, modeled on the decomposition $\K^2 \# \K^2$, is
easy to construct, while the second example, modeled on the
decomposition $\K^2 \# \T^2$, yields a better systolic area.  Combined
with Theorem~\ref{BB}, we deduce Proposition~\ref{p51} below.

\medskip

First we present an easy construction to bound the systolic area by~$1.5$.  

\begin{example}
Consider a unit square with suitable identification on the boundary
making it into a Klein bottle.  Cut out a little square of side length
$\frac{1}{2}$ at the center, with the sides parallel to those of the
unit square, and form the double along the boundary of the little
square.  This gives a piecewise flat surface~$M$ with four conical
singularities of angle~$3\pi$ homeomorphic to~$\K^2\#\K^2=4\RP^2$.
Observe that all the geodesics parallel to the sides of the unit
square are systolic closed curves.  By construction, the surface~$M$
has unit systole and area~$\frac{3}{2}$, i.e., $\sigma(M)=1.5$.
\end{example}

The following proposition provides a sharper estimate.

\begin{proposition}
\label{p51}  
There exists a nonpositively curved piecewise flat metric on the
Klein-double with the following properties:
\begin{enumerate}
\item
the metric is glued from four copies of an inscribed Euclidean
hexagon and a cylinder;
\item
it has  two conical singularities of angle $4(\pi- \theta)
\simeq 460^\circ$ and eight conical singularities of
angle~$2\pi+\theta \simeq 425^\circ$, where
\[
\theta = 2 \arccos\left(\frac{1+\sqrt{33}}{8}\right) \simeq 65^\circ;
\]
\item
the resulting systolic area satisfies
\[
\sigma(4\RP^2) = 1 + \frac{1}{16} \sqrt{414-66 \sqrt{33}}  \simeq
1.3690;
\]
\end{enumerate}

Thus, the minimal systolic area among all nonpositively curved metrics
on~$4\RP^2$ lies between~$1.002$ and~$1.37$.
\end{proposition}

\forget

\begin{proposition}
The minimal systolic area of $4\RP^2$ is in the interval
$[1.002,1.5]$.
\end{proposition}

\begin{proof}
The lower bound follows from Theorem~\ref{BB}.  To establish the upper
bound, consider a unit square with suitable identification on the
boundary making it into a Klein bottle.  We cut out a little square of
side $\frac12$ at the center and form the double along the boundary of
the little square.  This gives a $4\RP^2$ of area $2 - \frac12 = 1.5$
and unit systole.
\end{proof}

\forgotten


The construction is inspired by a construction of~\cite{KS15}.  

\subsection{Inscribed Euclidean hexagon}
\label{s51}

\mbox{ } \medskip

We first introduce a non-regular inscribed Euclidean hexagon defined as
follows.  Consider the symmetric Euclidean hexagon $H$ composed of
pairwise opposite isosceles triangles based at the center: four of
them have height~$\frac{1}{4}$ and main angle~$\theta$, and two of
them have height~$h$ and base~$\frac{1}{4}$; see Figure~\ref{fig:H}.
Note that all the radii arising from the center of~$H$ are equal, and therefore the hexagon is inscribed.

\begin{figure}[htbp!] 
\vspace{1.5cm}
\def\svgwidth{3.5cm}
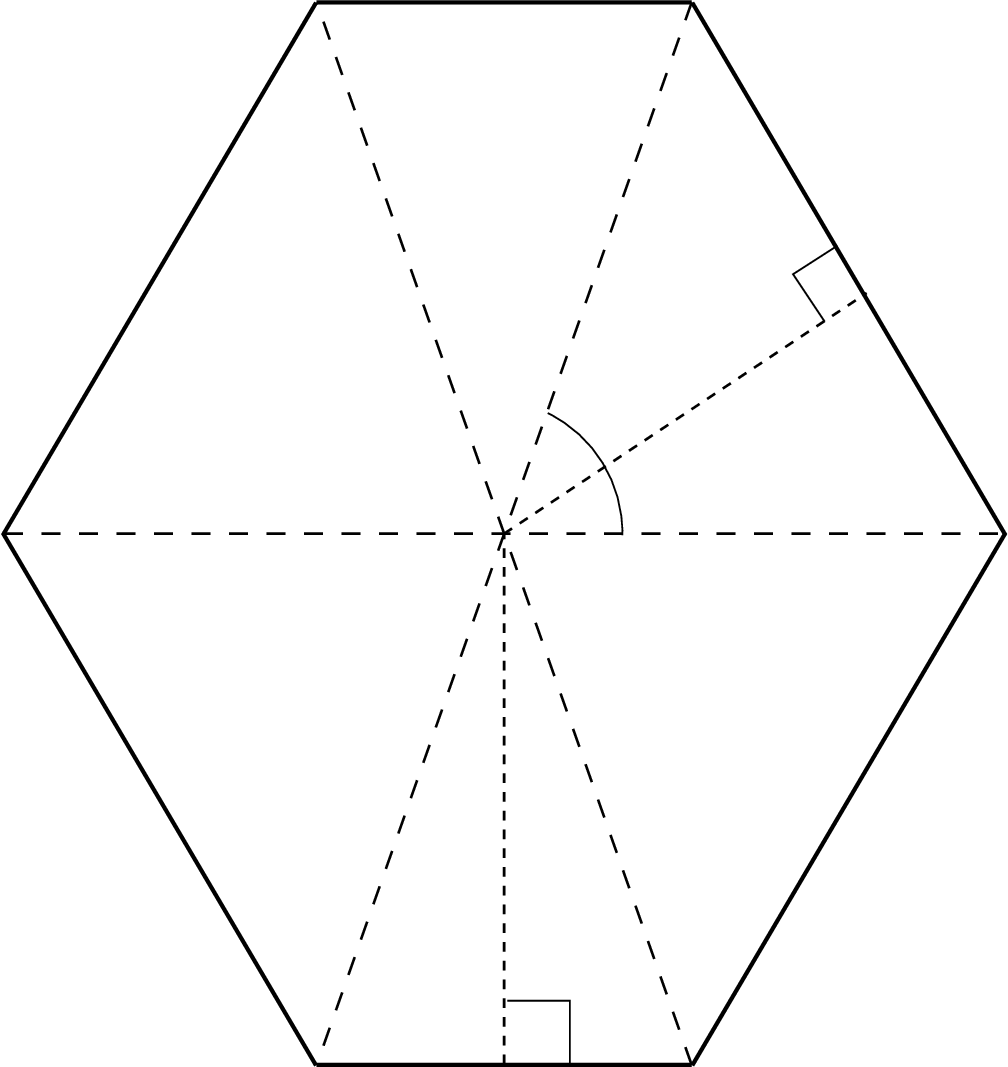 
\vspace{0.2cm}
\caption{The inscribed hexagon~$H$} \label{fig:H}
\end{figure}

In order for the isosceles triangles to have the same leg length, the angle~$\theta$ and the height~$h$ must satisfy
\begin{equation} \label{eq:theta-h}
\left\{
\begin{array}{rl}
\theta & \displaystyle = 2 \arccos\left(\frac{1+\sqrt{33}}{8}\right) \simeq 65.07^\circ \\
 & \\
h & \displaystyle   = \frac{1}{8} \sqrt{\frac{15-\sqrt{33}}{2}} \simeq 0.2689
\end{array}
\right.
\end{equation}

The base of the isosceles triangles with main angle~$\theta$ has length $\frac{1}{2} \tan(\frac{\theta}{2}) \simeq 0.3189$.
We will refer to these sides of~$H$ as the \emph{long sides}, and to the sides of length~$\frac{1}{4}$ as the \emph{short sides}.
Thus, the inscribed hexagon~$H$ has four long sides and two short sides.

The area of~$H$ is equal to
\[
\area(H) = \frac{1}{32} \sqrt{\frac{111-\sqrt{33}}{2}} \simeq 0.2267.
\]

\subsection{Hyperelliptic genus three surface with symmetries}

\mbox{ } \medskip

In order to describe the surface in Proposition~\ref{p51}, it is
convenient to start with the genus three hyperelliptic Riemann
surface~$\Sigma_3$ with binary octahedral symmetry group given by the
following lemma.  Although this surface is known to experts (see e.g.,
\cite{Ro97}), we include a description of its construction for the
reader's convenience.

\begin{lemma} \label{lem:Sigma3}
The genus three Riemann surface~$\Sigma_3$ defined as the smooth completion of the affine algebraic curve 
\[
y^2=x^8-14x^4+1
\]
admits a conformal hyperelliptic involution inducing a two-fold ramified cover $\Sigma_3 \to S^2$ with eight branch points corresponding to the centers of the faces of a regular octahedral decomposition of~$S^2$.
\end{lemma}

\begin{proof}
We will define $\Sigma_3$ by an equation $y^2=P(x)$, where $P$ is a monic polynomial of degree~$8$ whose zeros are located at the vertices of a cube inscribed in the Riemann sphere.
Up to rotations, we can assume that the inscribed cube has horizontal and vertical faces.
Denote by $A$, $B$, $C$, $D$ the vertices in the North hemisphere~$S^2_+$ and by $A'$, $B'$, $C'$, $D'$ the symmetric vertices in the South hemisphere.
Up to rotations, we can further assume that $A$ lies in the vertical plane containing the $x$-axis of~$\R^2$ and has positive abscisse.
From the relation $AB=AA'$, we deduce that the angle~$\varphi$ between~$OA$ and the horizontal plane satisfies $\sqrt{2} \cos \varphi = 2 \sin \varphi$.
The stereographic projection from the North pole sends~$A=(\cos \varphi,0,\sin \varphi)$ to the point $\sqrt{2+\sqrt{3}} \in \C$.
By rotations, this projection sends $A$, $B$, $C$, $D$ to the roots of $z^4 - (\sqrt{2+\sqrt{3}})^4$.
Taking an inversion with respect to the unit circle, it also sends $A'$, $B'$, $C'$, $D'$ to the roots of $z^4 - (\sqrt{2-\sqrt{3}})^4$.
Thus, the Riemann surface~$\Sigma_3$ is defined as the smooth completion of the affine algebraic curve given by the product of these two polynomials, that is,
\[
y^2=x^8-14x^4+1.
\]

One can check from the Riemann--Hurwitz formula that this equation defines a genus three surface.
\end{proof}

\subsection{Gluing inscribed Euclidean hexagons}
\label{s53}

\mbox{ } \medskip



Consider the two-fold ramified cover $\Sigma_3 \to S^2$ given by Lemma~\ref{lem:Sigma3} whose branch points correspond to the centers of the faces of a regular octahedral decomposition of~$S^2$, with vertex the North pole of~$S^2$.
The North hemisphere $S^2_+$ is formed by the triangular faces adjacent
to the North pole.  Under the ramified cover, this hemisphere lifts to a genus
one surface~$\Sigma_{1,2}$ with two disks removed.  The
surface~$\Sigma_{1,2}$ inherits a hexagonal decomposition formed of
four hexagons corresponding to the lifts of the triangles in the
triangulation of~$S^2_+$.  Each hexagon has exactly two boundary
edges, which are opposite sides of the hexagon and belong to different
boundary components of~$\Sigma_{1,2}$.

While retaining the topological structure, we now modify the conformal
class as follows.  We identify each hexagon with the Euclidean
hexagon~$H$ of Section~\ref{s51} so that the short sides correspond to
these two opposite boundary edges.  With these identifications, each
boundary component of~$\Sigma_{1,2}$ has length~$1$.

By construction, the hyperelliptic involution $\iota$ acts on each
hexagon by the central symmetry and exchanges the two boundary
components of~$\Sigma_{1,2}$.
Next, we use $\iota$ to glue together the connected components of the
boundary of~$\Sigma_{1,2}$ after inserting a cylinder as follows.

Denote one of the boundary components by $\mathcal B$.  Let
$I=[0,1-2h]$, where $h$ is defined in~\eqref{eq:theta-h}.  We
identify~$\mathcal B$ with the boundary component ${\mathcal B} \times
\{0\}$ of the flat cylinder~$\mathcal{C}= {\mathcal B} \times I$ of
unit circumference and height~$1-2h$.  We now identify ${\mathcal B}
\times\{1-2h\}$ with the other boundary component of $\Sigma_{1,2}$ by
the map
\[
(p,1-2h)\mapsto\iota(p).
\]
The resulting surface
\[
M= \Sigma_{1,2} \cup \mathcal{C}
\]
is piecewise flat and homeomorphic to~$4\RP^2=\T^2 \# \K^2$.  By
construction, it contains the cylinder $\mathcal{C} \subseteq M$.

The surface~$M$ has two conical singularities located at the preimages
of the North pole under the ramified double cover.  Four copies
of~$H$ meet around each of these singularities at an angular sector
delimited by two long sides.  Since the angle of each of these sectors
is equal to $\pi-\theta$, the two conical singularities have an angle
of~$4(\pi- \theta)$, where $\theta$ is defined in~\eqref{eq:theta-h}.

The surface~$M$ has eight additional conical singularities
corresponding to the four vertices along the equator of the
triangulation of the hemisphere (four on each boundary component of
the flat cylinder $\mathcal{C} \subseteq M)$.  Two copies of~$H$ meet around each
of these singularities at an angular sector delimited by a short and a
long side, along the flat cylinder~$\mathcal{C}$.  Since the angle of each of
these sectors is equal to $\frac{1}{2}(\pi+\theta)$, each of the eight
conical singularities has an angle of~$2\pi+ \theta$, with a
contribution of~$\pi$ from the flat cylinder~$\mathcal{C}$.

One can check that the surface~$M$ satisfies the Gauss--Bonnet equation
\[
2\pi \chi(M) = \sum_{i=1}^k (2\pi-\alpha_i)
\]
for piecewise flat metrics with $k$ conical singularities of angles~$\alpha_i$.

\medskip


\subsection{Systolic loops and area}

\mbox{ } \medskip

Let $M$ be the surface constructed in Section~\ref{s53}.

\begin{lemma}
The systole of~$M$ equals~$1$.  Furthermore, the surface~$M$ admits
the following three families of systolic closed curves:
\begin{enumerate}
\item two-sided closed geodesics orthogonal to the long sides of the
  hexagon; \label{f1}
\item one-sided closed geodesics orthogonal to the short sides of the
  hexagon at their midpoints;
\item two-sided closed geodesics foliating the cylinder~$\mathcal{C}$.
\end{enumerate}
The systolic closed curves in the first family are parallel to the inverse image under the ramified cover $\Sigma_{1,2} \to S^2_+$ of the segment, for the quotient metric, joining a pair of branch points.
\end{lemma}

\begin{proof}
By construction, the surface~$M$ decomposes into~$\Sigma_{1,2}$, consisting of four copies of~$H$, and the cylinder~$\mathcal{C}$.
The cylinder~$\mathcal{C}$ has unit systole, hence the systole of~$M$ is at most~$1$.
Furthermore, every geodesic arc of~$\mathcal{C}$ with endpoints on~$\partial \mathcal{C}$ is of length at least the height~$1-2h$ of the cylinder.
Similarly, every geodesic arc of~$\Sigma_{1,2}$ with endpoints on~$\partial \Sigma_{1,2}$ is of length at least~$2h$.
Apart from the closed geodesics of unit length foliating~$\mathcal{C}$, there are exactly four closed geodesics of unit length intersecting~$\mathcal{C}$, namely, the one-sided closed geodesics orthogonal to the short sides of the hexagon at their midpoint.
All the other noncontractible loops intersecting~$\mathcal{C}$ have length greater than~$1$.

Consider now a systolic loop
\[
\gamma\subseteq\Sigma_{1,2}.
\]
Since the hyperelliptic involution~$\iota$ on~$\Sigma_{1,2}$ induces minus the identity homomorphism in homology, the two simple loops~$\gamma$ and~$-\iota(\gamma)$ are homologous.

Suppose that the two simple curves~$\gamma$ and~$-\iota(\gamma)$ are homotopic.
By the flat strip theorem~\cite[\S II.2.13]{BH}, the two systolic loops~$\gamma$ and~$-\iota(\gamma)$ bound a flat cylinder.
This cylinder is invariant under~$\iota$ and contains exactly two fixed points of~$\iota$, which lie in its median circle in antipodal positions.
Since these two points are at distance~$\frac{1}{2}$ from each other, it follows that the length of~$\gamma$ is~$1$.
This situation occurs when $\gamma$ is parallel to the closed geodesic made of the two segments of length~$\frac{1}{2}$ joining the center of two adjacent hexagons.
In this case, $\gamma$ is a closed geodesic orthogonal to the long sides of the hexagon.

Suppose that the two simple curves~$\gamma$ and~$-\iota(\gamma)$ are homologous but not homotopic.
Then there exist two simple loops homotopic to~$\gamma$ and~$-\iota(\gamma)$ which bound a surface of positive genus.
This implies that $\gamma$ is homotopic to one of the boundary components of~$\Sigma_{1,2}$.
Since the boundary components of~$\Sigma_{1,2}$ are geodesic, and even uniquely geodesic in their homotopy classes, the curve~$\gamma$ lies in~$\partial \Sigma_{1,2} \simeq \partial \mathcal{C}$ and is of length~$1$. 
\end{proof}

It follows that the systolic area of~$M$ coincides with its area.
That is,
\begin{align*}
\sigma(M) & = 4 \area(H) + \area(\mathcal{C}) \\
 & = \frac{1}{8} \sqrt{\frac{111-\sqrt{33}}{2}} + 1 - \frac{1}{4} \sqrt{\frac{15-\sqrt{33}}{2}} \\
 & = 1 + \frac{1}{16} \sqrt{414-66 \sqrt{33}} \simeq 1.3690.
\end{align*}

\begin{remark}
There are 16 connected components of the region on $M$ free of
systolic loops.  Each of them is a right-angle triangle based on half
a short edge of a hexagon of type $H$.  The angle (at the base) of the
right-angle systole-free triangle (with base $\frac18$) is
$\frac{\theta}2$.  Note that in the nonpositive curvature case,
systole-free regions do appear for the systolically extremal metric
on~$3\RP^2$.
\end{remark}

\end{document}

%% file: geodesics.eps_tex
\begingroup%
  \makeatletter%
  \providecommand\color[2][]{%
    \errmessage{(Inkscape) Color is used for the text in Inkscape, but the package 'color.sty' is not loaded}%
    \renewcommand\color[2][]{}%
  }%
  \providecommand\transparent[1]{%
    \errmessage{(Inkscape) Transparency is used (non-zero) for the text in Inkscape, but the package 'transparent.sty' is not loaded}%
    \renewcommand\transparent[1]{}%
  }%
  \providecommand\rotatebox[2]{#2}%
  \newcommand*\fsize{\dimexpr\f@size pt\relax}%
  \newcommand*\lineheight[1]{\fontsize{\fsize}{#1\fsize}\selectfont}%
  \ifx\svgwidth\undefined%
    \setlength{\unitlength}{789.27163696bp}%
    \ifx\svgscale\undefined%
      \relax%
    \else%
      \setlength{\unitlength}{\unitlength * \real{\svgscale}}%
    \fi%
  \else%
    \setlength{\unitlength}{\svgwidth}%
  \fi%
  \global\let\svgwidth\undefined%
  \global\let\svgscale\undefined%
  \makeatother%
  \begin{picture}(1,0.70707072)%
    \lineheight{1}%
    \setlength\tabcolsep{0pt}%
    \put(0,0){\includegraphics[width=\unitlength]{geodesics.eps}}%
    \put(-0.13319557,1.03381442){\makebox(0,0)[lt]{\lineheight{1.25}\smash{\begin{tabular}[t]{l}$\bar{\eta}_0$\end{tabular}}}}%
    \put(0.21592148,0.45805372){\makebox(0,0)[lt]{\lineheight{1.25}\smash{\begin{tabular}[t]{l}$\bar{\nu}_0$\end{tabular}}}}%
    \put(0.68560551,0.45805372){\makebox(0,0)[lt]{\lineheight{1.25}\smash{\begin{tabular}[t]{l}$\bar{\nu}_t$\end{tabular}}}}%
    \put(1.04356225,1.0309185){\makebox(0,0)[lt]{\lineheight{1.25}\smash{\begin{tabular}[t]{l}$\bar{\eta}_t$\end{tabular}}}}%
    \put(0.46809569,-0.10103982){\makebox(0,0)[lt]{\lineheight{1.25}\smash{\begin{tabular}[t]{l}$x_0$\end{tabular}}}}%
    \put(0.15346562,0.25030796){\makebox(0,0)[lt]{\lineheight{1.25}\smash{\begin{tabular}[t]{l}$\frac{1}{2} \sys$\end{tabular}}}}%
    \put(0.12617332,0.69504924){\makebox(0,0)[lt]{\lineheight{1.25}\smash{\begin{tabular}[t]{l}$h$\end{tabular}}}}%
  \end{picture}%
\endgroup%

%% file: hexagon.eps_tex
\begingroup%
  \makeatletter%
  \providecommand\color[2][]{%
    \errmessage{(Inkscape) Color is used for the text in Inkscape, but the package 'color.sty' is not loaded}%
    \renewcommand\color[2][]{}%
  }%
  \providecommand\transparent[1]{%
    \errmessage{(Inkscape) Transparency is used (non-zero) for the text in Inkscape, but the package 'transparent.sty' is not loaded}%
    \renewcommand\transparent[1]{}%
  }%
  \providecommand\rotatebox[2]{#2}%
  \newcommand*\fsize{\dimexpr\f@size pt\relax}%
  \newcommand*\lineheight[1]{\fontsize{\fsize}{#1\fsize}\selectfont}%
  \ifx\svgwidth\undefined%
    \setlength{\unitlength}{789.27163696bp}%
    \ifx\svgscale\undefined%
      \relax%
    \else%
      \setlength{\unitlength}{\unitlength * \real{\svgscale}}%
    \fi%
  \else%
    \setlength{\unitlength}{\svgwidth}%
  \fi%
  \global\let\svgwidth\undefined%
  \global\let\svgscale\undefined%
  \makeatother%
  \begin{picture}(1,0.70707072)%
    \lineheight{1}%
    \setlength\tabcolsep{0pt}%
    \put(0,0){\includegraphics[width=\unitlength]{hexagon.eps}}%
    \put(0.62466971,0.589897){\makebox(0,0)[lt]{\lineheight{1.25}\smash{\begin{tabular}[t]{l}$\theta$\end{tabular}}}}%
    \put(0.72967101,0.66215365){\makebox(0,0)[lt]{\lineheight{1.25}\smash{\begin{tabular}[t]{l}$\frac{1}{4}$\end{tabular}}}}%
    \put(0.41331049,0.16964818){\makebox(0,0)[lt]{\lineheight{1.25}\smash{\begin{tabular}[t]{l}$h$\end{tabular}}}}%
    \put(0.45419696,-0.11801057){\makebox(0,0)[lt]{\lineheight{1.25}\smash{\begin{tabular}[t]{l}$\frac{1}{4}$\end{tabular}}}}%
  \end{picture}%
\endgroup%